\documentclass[a4paper, reqno]{amsproc}

\usepackage{amsfonts,amssymb}
\usepackage{graphicx}
\usepackage{epstopdf}

\newtheorem{theorem}{\rm\bf Theorem}[section]

\newtheorem*{theorem 1}{\rm\bf Proposition 1}
\newtheorem*{theorem 2}{\rm\bf Proposition 2}
\newtheorem*{mainlemma*}{\rm\bf Main Lemma}

\theoremstyle{definition}
\newtheorem{definition}[theorem]{\rm\bf Definition}

\theoremstyle{remark}
\newtheorem{remark}[theorem]{\rm\bf Remark}

\newtheorem{example}[theorem]{\rm\bf Example}

\def\scal#1#2{\langle #1, #2\rangle}
\def\R#1{\mathbb{R}^{#1}}

\def\TOP{\top}

\DeclareMathOperator{\re}{\mathrm{Re}}

\DeclareMathOperator{\trace}{\mathrm{tr}}

\def\Nabla{\overline{\,\nabla}}

\begin{document}
\author{Tkachev Vladimir G.}
\title{A note on isoparametric polynomials}
\address{V. G. Tkachev\\
    Link\"oping University, Department of Mathematics, SE-581 83          }
\email{tkatchev@kth.se}
%\date{\today}

\begin{abstract}
We show that any homogeneous polynomial solution of $|\nabla F(x)|^2=m^2|x|^{2m-2}$, $m\ge 1$, is either  a radially symmetric polynomial $F(x)=\pm |x|^{m}$ (for even $m$'s) or it is a composition of a Chebychev polynomial and a Cartan-M\"unzner polynomial.
\end{abstract}
\maketitle

\section{Introduction}

In this paper we study homogeneous polynomials $F(x)$ of degree $m\ge1$ in  $x\in \R{n}$ satisfying the following  equation:
\begin{equation}\label{Munz1}
|\nabla F(x)|^2=m^2|x|^{2m-2}, \quad m\in \mathbb{N}=\{1,2,3,\ldots\}.
\end{equation}
We make no distinction between two solutions $F_1$ and $F_2$ which differ in the sign, and write in such a  case $F_1\cong F_2$.

For any even $m\ge 2$ there exists an obvious solution $F(x)=|x|^m$. Despite the fact that the right hand side of (\ref{Munz1}) is radially symmetric, there are plenty other, non-radially symmetric polynomial solutions. For instance, it is easily shown that for $m=2$ any quadratic solution of (\ref{Munz1}) is of the form
\begin{equation}\label{g2}
F(x)=|V(x)|^2-|V^\bot (x)|^2=2|V(x)|^2-|x|^2
\end{equation}
for some subspace $V\subset \R{n}$, where $V(x)$ denotes the orthogonal projection of $x$ onto $V$. This family is well-known by its connection to the so-called isoparametric hypersurfaces in the Euclidean spheres which are crucial for the further classification. We summarize  below some facts and definitions from the isoparametric hypersurfaces theory that will be used throughout the paper.

Recall that a hypersurface of the unit sphere $S^{n-1}\subset\R{n}$ is called \textit{isoparametric} if it has constant principal curvatures. More precisely, if $2\le \dim V\le n-2$ then any level set $F^{-1}(t)\cap S^{n-1}$, $t\in (-1,1)$, is a smooth hypersurface  in $S^{n-1}$ with exactly two distinct constant principal curvatures. It was Cartan \cite{Cartan39MZ} who was first to show that any isoparametric hypersurface in $S^{n-1}$ with \textit{three} distinct principal curvatures is the restriction on the unit sphere of a level set of a \textit{harmonic} cubic solution of (\ref{Munz1}) for $m=3$, and vice versa. He also  showed that there exist (up to rotation in $\R{n}$) exactly four harmonic cubics satisfying (\ref{Munz2}), each one in dimensions $n=5,8,14$ and $26$, and given explicitly by
\begin{equation}\label{CartanFormula}
\begin{split}
F(x) =x_{n}^3&+\frac{3}{2}x_{n}(|z_1|^2+|z_2|^2-2|z_3|^2-2x_{n-1}^2)\\
&+\frac{3\sqrt{3}}{2}x_{n-1}(|z_2|^2-|z_1|^2)+{3\sqrt{3}}\re (z_1z_2z_3),
\end{split}
\end{equation}
where $z_k=(x_{kd-d+1},\ldots,x_{kd})\in \R{d}=\mathbb{F}_d$, $k=1,2,3$, and  $\mathbb{F}_d$ denotes the real division algebra of dimension $d$: $\mathbb{F}_1=\mathbb{R}$ (the reals), $\mathbb{F}_2=\mathbb{C}$ (the complexes), $\mathbb{F}_4=\mathbb{H}$ (the quaternions) and $\mathbb{F}_8=\mathbb{O}$ (the octonions). In the middle of 1970's, F.~M\"unzner \cite{Mun1} generalized the Cartan results and proved that the number of distinct principal curvatures of an isoparametric hypersurface in a unit sphere can only be $m=1,2,3,4$ or $6$, and that any such a hypersurface is always obtained as the restriction on the unit sphere of a level set of some homogeneous polynomial solution of (\ref{Munz1}) satisfying additionally to
\begin{equation}\label{Munz2}
\Delta F(x)=c|x|^{m-2},
\end{equation}
with $c\in \R{}$ ($c=0$ when $m$ is odd) such that
\begin{equation}\label{pm}
m_{\pm}=\frac{n-2}{m}\pm\frac{c}{m^2}\in \mathbb{N}.
\end{equation}
Moreover, M\"unzner established that all the level sets $F^{-1}(t)\cap S^{n-1}$, $t\in (-1,1)$, are connected isoparametric hypersurfaces whose principal curvatures have multiplicities either $m_+$ or $m_-$.

\begin{definition}
A polynomial solution of (\ref{Munz1}) and (\ref{Munz2}) with $c$ satisfying (\ref{pm}) is called a \textit{Cartan-M\"unzner polynomial}, or just a \textit{CM-polynomial}.
\end{definition}

Some important  examples of CM-polynomials are listed in Appendix~1 below.
\begin{remark}
The role the condition (\ref{pm}) plays in the above definition is seen  from the following observations. Note that any radially symmetric function $|x|^{2k}$, $k\in \mathbb{N}$ provides a  polynomial satisfying both (\ref{Munz1}) and (\ref{Munz2}) but not (\ref{pm}) because one has $m_-=-1\not\in \mathbb{N}$ in this case. On the other hand,  M\"unzner  proved \cite[p.~65-66]{Mun1} that any non-radially symmetric polynomial $\widetilde{F}$ of degree $\widetilde{m}$ satisfying both (\ref{Munz1}) and (\ref{Munz2}) is either a CM-polynomial, or there holds $\widetilde{F}=\pm(2F^2-|x|^{\widetilde{m}})$ for some CM-polynomial $F$.
The latter relation can also be written as
\begin{equation}\label{Munz3}
\widetilde{F}|_{S^{n-1}}\cong T_2(F|_S^{n-1}),
\end{equation}
where
$$
T_k(t)= \re \bigr(t+i\sqrt{1-t^2}\bigl)^k\equiv \cos (k \cos^{-1}t)
$$
denotes the $k$th Chebyshev polynomial of the first kind: $T_1(t)=t$, $T_2(t)=2t^2-1$, and $T_n(t)=2tT_{n-1}(t)-T_{n-2}(t).$
\end{remark}

Now we return again to solutions of eq. (\ref{Munz1}) alone. In  \cite{TkCartan}, we proved that any cubic solution of (\ref{Munz1}) is either a cubic Cartan-M\"unzner polynomial (\ref{CartanFormula}) or the irreducible cubic form
\begin{equation}\label{rn}
F(x)=x_{n}^3-3x_{n}(x_1^2+\ldots x_{n-1}^2), \quad x\in \R{n}, n\ge 2.
\end{equation}
A similar situation holds true also for  $m=4$. More precisely, any quartic solution $F$ of (\ref{Munz1}) of degree four is either a CM-polynomial or given by
\begin{equation}\label{oddm4}
\begin{split}
F(x)= |V(x)|^4-6|V(x)|^2|V^\bot(x)|^2+|V^\bot(x)|^4
\end{split}
\end{equation}
for some subspace $V\subset \R{n}$, see \cite{TkIso4}. In the latter paper, we observed that both (\ref{rn}) and (\ref{oddm4}) enjoy the common representation
\begin{equation}\label{virtue}
\re (|V(x)|+ |V^\bot(x)|\sqrt{-1})^m=\sum_{k=0}^{[m/2]}(-1)^k\binom{m}{k}|V(x)|^{m-2k} |V^\bot(x)|^{2k}
\end{equation}
for some subspace $V\subset \R{n}$, where $|V(x)|=V(x)$ whenever $\dim V=1$. For instance,  (\ref{virtue}) agrees with (\ref{rn}) for $V=\R{}e_n$. It  is easily verified that any polynomial given by (\ref{virtue}) provides a solution of (\ref{Munz1}). We have conjectured in \cite{TkIso4} that any solution of (\ref{Munz1}) is either a CM-polynomial or a solution given by (\ref{virtue}). The following example shows, however,  that this is not the case.

\begin{example}
Let $F$ be the  Cartan polynomial (\ref{CartanFormula}) in $\R{5}$ and set $F_1(x)=|x|^6-2F^2(x)$. Then $F_1$ is easily shown to satisfy to (\ref{Munz1}) for $m=6$ and $n=5$.  On the other hand, a result of Abresh \cite{Abresh} states that isoparametric hypersurfaces with $m=6$ distinct principal curvatures can exist only in $S^7$ and $S^{13}$, i.e. the corresponding CM-polynomials are defined in $\R{8}$ and $\R{14}$. Thus, $F_1$ cannot be a CM-polynomial. We show that it cannot be also represented by (\ref{virtue}). Indeed, arguing by contradiction, we have that either $-F_1(x)$ or $F_1(x)$ is of the form
\begin{equation*}
\begin{split}
G(x)&:=\re (|V(x)|+ |V^\bot(x)|\sqrt{-1})^6\\
&=32|V(x)|^6-48|x|^2|V(x)|^4+18|x|^4|V(x)|^2-|x|^6
\end{split}
\end{equation*}
for some subspace $V\subset \R{n}$. If $-F_1(x)=G(x)$ then $F^2(x)=|V(x)|^2(16|V(x)|^4-24|x|^2|V(x)|^2+9|x|^4)$  readily yields a contradiction because $F(x)$ is irreducible. On the other hand, if $F_1(x)=G(x)$ then $F^2(x)=(|x|^2-V(x)^2)(|x|^2-4V(x)^2)^2$ yields again the reducibility of $F(x)$, a contradiction follows.

\end{example}

In this paper, we classify all polynomial solutions of (\ref{Munz1}) relating them to the CM-polynomials as primitive building blocks.

\begin{theorem}\label{th1}
For any homogeneous polynomial solution $F$ of $(\ref{Munz1})$ either of the following is true:
\begin{itemize}
\item[(i)]
$m$ is even and $F$ is radially symmetric, i.e. $F(x)\cong  |x|^{m}$;
\item[(ii)]
there exists an integer $p\in \{1,2,3,4,6\}$, $p|m$, and a CM-polynomial $G(x)$ of degree $p$ such that
\begin{equation}\label{composition}
F|_{S^{n-1}}\cong T_{k}\bigl(G|_{S^{n-1}}\bigr), \quad k=\frac{m}{p}.
\end{equation}
\end{itemize}
Moreover, CM-polynomials are primitive in  the sense that if two such polynomials $F$ and $G$ satisfy $(\ref{composition})$ then $F\cong G$.
\end{theorem}

\begin{remark}
Notice that  $F(x)=T_k(G(x))$ on $S^{n-1}$ implies  by homogeneity an identity in $\R{n}$
\begin{equation}\label{FG}
F(x)=|x|^m T_k\bigl(G(x)|x|^{-p}\bigr).
\end{equation}
\end{remark}

\begin{remark}
The alternative (i), the cases $p=1$ and $p=2$ in Theorem~\ref{th1} are readily seen yield the above representation (\ref{virtue}) for $(\dim V,m)=(n,m)$, $(\dim V,m)=(1,m)$, and $(\dim V,m)=(s,m/2)$ with $2\le s\le [n/2]$, respectively.
\end{remark}

\bigskip
We demonstrate below how (\ref{composition}) yields the particular cases $m=3$ and $m=4$ discussed above. First observe that for $m=3$, Theorem~\ref{th1} yields (ii) with only possible values  $p=1$ or $p=3$. In the latter case one has the four CM-polynomials (\ref{CartanFormula}), and in the former case one obtains from (\ref{composition}) $F|_{S^{n-1}}\cong T_{3}\bigl(G|_{S^{n-1}}\bigr)$, where $G$ is a CM-polynomial of degree one, i.e. one can assume without loss of generality that $G=x_n$. Since $T_3(t)=4t^3-3t$ we have by lifting (\ref{composition}) in $\R{n}$ that $F(x)\cong |x|^3(4x_n^3-3x_n|x|^2)=x_n(x_n^2-3\sum_{i=1}^{n-1}x_i^2)$, which yields (\ref{rn}).

Now suppose that $m=4$ and $F$ is not radially symmetric. Then the only possible values of $p$ in (ii) are $1$, $2$ and $4$. Again, $p=4$ yields immediately that $F$ is a CM-polynomial. If $p=2$ then $G$ in (\ref{composition}) is a CM-polynomial of degree two, i.e. is given by (\ref{g2}) with $V\subset \R{n}$ satisfying $2\le\dim V\le n-2$. Since $T_2(t)=2t^2-1$ one has by lifting (\ref{composition}) in $\R{n}$ that
$$
F(x)\cong 2G^2(x)-|x|^4=2(|V(x)|^2-|V^\bot(x)|^2)^2-(|V(x)|^2+|V^\bot(x)|^2)^2
$$
which yields exactly the representation given by (\ref{oddm4}). Finally, if $p=1$ then $G=x_n$ and $T_4(t)=8t^4-8t^2+1$ yield
$$
F(x)=8x_n^4-8x_n^2|x|^2+|x|^4\equiv |V(x)|^4-6 |V(x)|^2|V^\bot(x)|^2+|V^\bot(x)|^4
$$
which, again, agrees with (\ref{oddm4}) if one put $V=\R{}e_n$.

\begin{remark}
Our proof of Theorem~\ref{th1} make essentially  use of earlier geometrical results on isoparametric hypersurfaces and transnormal functions  due to M\"unzner and Wang. On the other hand, the assertion  of the theorem is purely algebraic and is essentially equivalent to that any non-radially symmetric polynomial solution of (\ref{Munz1}) is a composition of a Chebyshev polynomial with a solution of (\ref{Munz1}) having the property (\ref{Munz2}) (a primitive solution). It would be interesting to find an independent algebraic proof to this fact.
\end{remark}

\begin{remark}
The alternative (i)  in Theorem~\ref{th1} can also be thought as a composition formula in the spirit of (\ref{composition}) if one assigns $p=2$ to the radially symmetric polynomial $G(x)=|x|^2$ and set $R_k(t)=t^k$ in (\ref{composition}) instead of $T_k$. Interestingly, the appearance of exactly $R_k$ and $T_k$ in our theorem is reminiscent of a well-known result of Ritt \cite{Ritt} that two univariate polynomials commute under composition if and only if, they are, within a linear homeomorphism, either both powers $R_k(t)$, both Chebyshev polynomials $T_k(t)$, or both iterates of the same polynomial. It would be interesting also to learn if this fact has some proper explanation.
\end{remark}

\section{Proof of the main results}

\begin{proof}[Proof of Theorem~\ref{th1}]

Note that (\ref{Munz1}) and Euler's homogeneous function theorem yield
\begin{equation}\label{trans1}
|\Nabla f(x)|^2=|\nabla F(x)|^2-\scal{x}{\nabla F(x)}^2=m^2(1-f(x)^2),
\end{equation}
where  $f(x)=F(x)|_{S^{n-1}}$ and $\Nabla$ stands for the Levi-Chevita connection on $S^{n-1}$.

First notice that an assumption $f\equiv\mathrm{const}$ is equivalent by virtue of (\ref{trans1}) to that $f\equiv 1$ on $S^{n-1}$ which implies $F(x)\equiv |x|^m$, and therefore yields the alternative (i) in Theorem~\ref{th1}.
Therefore we suppose that $f\not\equiv\mathrm{const}$. Then (\ref{trans1}) shows that
$|\Nabla f(x)|^2=a(f)$ with $a(t):=m^2(1-t^2)$ being a smooth function, thus $f(x)$ is a transnormal function on the sphere $S^{n-1}$. By a theorem of Wang \cite{Wang} (see also \cite[p.~966]{Thorb} and \cite{Miyaoka00})  there exists an \textit{isoparametric} function $g(x)$ on $S^{n-1}$ and a real valued function $h$ defined on the image of $g$ such that
\begin{equation}\label{composition1}
f(x)=h(g(x)), \quad x\in S^{n-1}.
\end{equation}
Recall that a function on $S^{n-1}$ is called isoparametric if $|\Nabla f(x)|^2=a(f)$ for some $C^2$-function $b$, and moreover $\underline{\Delta} f=b(f)$ for some function $a$ continuous  on the image $I:=f(S^{n-1})$, where $\underline{\Delta}$ is the Laplace-Beltrami operator on $S^{n-1}$. By a theorem of M\"unzner \cite{Mun1}, any isoparametric function on the sphere is a composition of a univariate function and  (the restriction of) a CM-polynomial, say $G(x)$. Thus, one can without loss of generality assume that $g(x)=G(x)|_S$ in (\ref{composition1}) where $G(x)$ satisfies the Cartan-M\"unzner equations
\begin{equation}\label{Mun2}
\begin{split}
&|\nabla G(x)|^2=p^2|x|^{2p-2}, \qquad \Delta G(x)=c|x|^{p-2},
\end{split}
\end{equation}
where $p\in \{1,2,3,4,6\}$ is the degree of $G$.

Next, notice that (\ref{trans1})  implies  $-1\le f(x)\le 1$ for any $x\in S^{n-1}$  and also that  $f^{-1}(t)$ is a compact submanifold (in fact, a hypersurface) of $S^{n-1}$ for any $t\in (-1,1)$. It is easy to see also that $f^{-1}(t)$ is non-empty for each $t\in [-1,1]$. Indeed, suppose  $x_0\in S^{n-1}$  be a maximum point of $f(x)=F(x)|_{S^{n-1}}$. Then $\nabla F(x_0)=\lambda x_0$ for some real $\lambda$, hence $\lambda^2=m^2$  in virtue of $|x_0|=1$ and (\ref{Munz1}), and, moreover, applying the Euler homogeneous function theorem one has $mf(x_0)=mF(x_0)=\scal{\nabla F(x_0)}{x_0}=\lambda$. It follows that $f(x_0)=\mathrm{sign}\, \lambda$. If $f(x_0)=-1$ then $f(x)\equiv -1$ on $S^{n-1}$ and therefore $F(x)\equiv |x|^{m}$ in $\R{n}$, a contradiction. Thus, $\max_{x\in S^{n-1}}f(x)=1$. A similar argument shows that $\min_{x\in S^{n-1}}f(x)=-1$ which proves the claim.

Clearly, the above argument is also valid for $g(x)$. In fact,  M\"unzner  \cite[p.65-67]{Mun1} showed that $g^{-1}(r)$ is  a connected compact hypersurface of $S^{n-1}$ for any $r\in (-1,1)$ and  $g^{-1}(\pm1)$ are (connected) smooth submanifolds of codimension at least two.

Now observe that the first equation in (\ref{Mun2}) yields  $|\Nabla g(x)|^2=p^2(1-g(x)^2)$, therefore we obtain by (\ref{composition1})
\begin{equation}\label{rer}
|\Nabla f(x)|^2=h'(g(x))^2\cdot |\Nabla g(x)|^2=p^2|h'(g(x))|^2(1-g(x)^2),
\end{equation}
and taking into account (\ref{trans1}) and (\ref{composition1}), we arrive at
\begin{equation}\label{diffeq}
\frac{m^2}{1-t^2}=\frac{p^2h'(t)^2}{1-h^2(t)}.
\end{equation}
Since both $g$ and $f$ are real analytic functions (the restrictions of polynomials on the sphere), $h(t)$ is also a real analytic function which yields by (\ref{diffeq})
\begin{equation}\label{htt}
h(t)=\cos (\frac{m}{p}\arccos t +C), \quad t\in [-1,1]
\end{equation}
for some real $C$.

We claim that $h^2(\pm1)=1$. Indeed, suppose for instance that $s:=h(1)\ne \pm1$. Then by the above, $f^{-1}(s)$ is a compact hypersurface of $S^{n-1}$. On the other hand, (\ref{composition1}) and (\ref{htt}) yield
$$
f^{-1}(s)=\bigsqcup_{r\in R} g^{-1}(r),
$$
where $R$ is the (finite) set of solutions $h(r)=s$. But $1\in R$, hence $f^{-1}(s)$ contains a disjoint component $g^{-1}(1)$ of codimension at least 2, a contradiction.

In particular, we have $0=1-h^2(1)=\sin^2 C$ which yields $C=\pi l$ for some integer $l$. Similarly, $0=1-h^2(-1)=\sin^2(\frac{m}{p}\pi+\pi l)$ which immediately implies that $k:=\frac{m}{p}$ is integer and that $h(t)=\cos (\arccos t +\pi l)=(-1)^l T_k(t)$ is a Chebyshev polynomial.

It remains only to verify the last claim of the theorem. To this end we notice that if $F$ in (\ref{composition}) is a CM-polynomial then for any $t\in (-1,1)$ the level sets $f^{-1}(t)$ are connected isoparametric hypersurfaces with $m=\deg F$ distinct principal curvatures. On the other yields (\ref{composition}) shows that $f^{-1}(t)=g^{-1}\circ T_k^{-1}(t)$ is a disjoint union of $\deg T_k=k$ connected components consisting of level sets of $g$ (observe that $G$ is a CM-polynomial too), thus $k=1$ and $T_1(t)=t$. The theorem is proved completely.

\end{proof}

\bigskip

\section{Appendix. Some examples of CM-polynomials}

Recall that a CM-polynomial can have only degree $m=1,2,3,4$ or 6. We shortly outline all  so far known CM-polynomials below (the only remained unsettled case is in dimension $n=32$ of degree $m=4$), see  the recent lection notes \cite{ChStory} for a full story. CM-polynomial of degrees $m=1,2,3$ are the linear forms $F=\scal{x}{e}$ (with $e\in \R{n}$ being a unit vector), quadratic forms  (\ref{g2}) for some subspace $V\subset\R{n}$ of dimension $2\le \dim V\le [n/2]$, and Cartan cubics (\ref{CartanFormula}), respectively.

According to the recent classification \cite{CJ}, \cite{Ch1}, \cite{Ch2}   any CM-polynomial of degree four in dimension $n\ne 32$ is  either one of the constructed in \cite{OT1}, \cite{OT2}, \cite{FKM} and given explicitly by
$$
F(x)=|x|^4-2\sum_{i=0}^{s}(x^\TOP A_i x)^2,
$$
where $\{A_{i}\}_{0\le i\le s}$ is a system of symmetric endomorphisms of $\R{n}$ satisfying
$A_i A_j+A_jA_i=2\delta_{ij}\cdot \mathbf{1}_{\R{n}}$, or one of the following two quartic forms
$$
F_d(x)=\frac{1}{2}(\trace (Z\bar Z)^2-\frac{3}{8}(\trace Z\bar Z)^2),\quad d=1,2,
$$
where the matrix 
$$
Z=\left(
    \begin{array}{ccccc}
      0 & z_1 & z_2 & z_3 & z_4 \\
      -z_1 & 0 & z_5 & z_6 & z_7 \\
      -z_2 & -z_5 & 0 & z_8 & z_9 \\
      -z_3 & -z_6 & -z_8 & 0 & z_{10} \\
      -z_4 & -z_7 & -z_9 & -z_{10} & 0 \\
    \end{array}
  \right)
$$
has entries $z_k=x_k$ if $d=1$ and $z_k=x_k+ix_{10+k}$ if $d=2$, and $1\le k\le 10$. 

Finally, there are two (harmonic) CM-polynomials of degree $m=6$, each one in dimensions $n=8$ and $n=14$, \cite{Abresh}, \cite{DorfNehar85}, \cite{Miyaoka12}. 

\bigskip

\textbf{Acknowledgments.} 
The author thanks Professor Tang Zizhou for bringing his attention to the paper  \cite{Wang} of Qi~Ming~Wang.

\bibliographystyle{amsalpha}
%\bibliography{}
\bibliography{main_references}

\end{document}